\documentclass[11pt]{amsart}
\usepackage{amssymb,mathrsfs,color,bbold,comment,centernot}
\usepackage[all]{xy}

\theoremstyle{plain}
\newtheorem{thm}{Theorem}[section]

\newtheorem{cor}[thm]{Corollary}

\theoremstyle{definition}
\newtheorem{remark}[thm]{Remark}

\newcommand{\abs}[1]{\lvert#1\rvert}
\newcommand{\norm}[1]{\lVert#1\rVert}
\newcommand{\bigabs}[1]{\bigl\lvert#1\bigr\rvert}
\newcommand{\bignorm}[1]{\bigl\lVert#1\bigr\rVert}

\newcommand{\la}{\langle}
\newcommand{\ra}{\rangle}
\newcommand{\bs}{\backslash}

\newcommand{\N}{{\mathbb N}}
\newcommand{\R}{{\mathbb R}}

\newcommand{\ep}{\varepsilon}

\newcommand{\al}{\alpha}

\newcommand{\bP}{{\mathbb P}}

\DeclareMathOperator{\Span}{Span}

\def\one{\mathbb 1}
\numberwithin{equation}{section}

\author[N.~Gao]{Niushan Gao}
\address{Department of Mathematics, Ryerson University, 350 Victoria Street, Toronto, Canada M5B2K3}
\email{niushan@ryerson.ca}

\author[D.~Leung]{Denny H.~Leung}
\address{Department of Mathematics, National University of Singapore, Singapore 117543}
\email{matlhh@nus.edu.sg}

\author[F.~Xanthos]{Foivos Xanthos}
\address{Department of Mathematics, Ryerson University, 350 Victoria Street, Toronto, Canada M5B2K3}
\email{foivos@ryerson.ca}

\keywords{Local Hahn-Banach Theorem, local convexity,  separating dual space, uo-dual,  uniform integrability}

\subjclass[2010]{46A22, 46A20, 60A10}
\thanks{The first and third authors acknowledge support of NSERC Discovery Grants. The second author is partially supported by AcRF grant R-146-000-242-114.}

\date{\today}

\begin{document}

\title[Local HB Theorem]{A local Hahn-Banach theorem and its applications}
 \maketitle

\begin{abstract}
An important consequence of the Hahn-Banach Theorem says that  on any locally convex Hausdorff topological space $X$, there are sufficiently many continuous linear functionals to separate points of $X$.
In the paper, we establish a ``local'' version of this theorem.  The result is applied  to study the uo-dual of a Banach lattice that was recently introduced in \cite{GLX:18}. We also provide a simplified approach to the measure-free characterization of uniform integrability established in \cite{K:14}.
\end{abstract}

\section{Introduction}

The following fundamental result is a well-known consequence of the Hahn-Banach Theorem.

\begin{thm}[Hahn-Banach]\label{HB}
Let $(X,\tau)$ be a real, locally convex, Hausdorff topological vector space. Let $X^*$ be the collection of $\tau$-continuous linear functionals on $X$. Then $X^*$ separates points of $X$, namely, for any $x\in X\bs\{0\}$, there exists $\phi\in X^*$ such that $\phi(x)\neq 0$.
\end{thm}

Theorem \ref{HB} establishes $\langle X, X^*\rangle$ as a dual pair, thus enabling powerful results of duality theory to be brought to bear in many situations.  To a large extent, it  explains the central importance of the Hahn-Banach Theorem in many applications.
The present paper aims at proving the following localized version, with applications to follow.

\begin{thm}\label{localHB}
Let $(X,\tau)$ be a real, Hausdorff  topological vector space. Let $B$ be a convex, absorbing, balanced set in $X$ and $X^*_B$ be the collection of all linear functionals on $X$ that are continuous on $B$ endowed with the relative $\tau$-topology. Suppose that the relative $\tau$-topology on $B$ is locally convex. Then $X^*_B$ separates points of $X$.
\end{thm}

Clearly, taking $B=X$ in Theorem~\ref{localHB} yields Theorem~\ref{HB}. In the case where $X$ is a normed space, $\tau$ is the norm topology, and $B$ is the unit ball, $X_B^*$ is just the norm dual of $X$. In the case where $X$ is the norm dual of a Banach space $Y$, $\tau=\sigma(X,Y)$, and $B$ is the unit ball of $X$, then $X_B^*=Y$, by virtue of the Krein-Smulian theorem.

More interesting applications of Theorem~\ref{localHB} take place when the vector topology $\tau$ is not locally convex itself. Indeed, after proving Theorem~\ref{localHB} in Section~\ref{seclhb}, we will apply it to study the uo-dual of Banach function spaces and Banach lattices, which was recently introduced and studied in \cite{GLX:18}; see Theorems~\ref{t8} and \ref{uod-oc} below. Another application that demonstrates the power of Theorem~\ref{localHB} is Corollary~\ref{t9} where we give a simple proof of a recent theorem on measure-free characterizations of uniformly integrability due to Kardaras \cite{K:14}.

We first establish some notation and terminology. For a topology $\tau$ on a set $X$, denote by $\tau_B$ the relative $\tau$-topology on a subset $B$ of $X$. A subset $B$ of a real vector space $X$ is said to be \emph{balanced} if $\lambda B\subset B$ whenever $\abs{\lambda}\leq 1$ and is said to be \emph{absorbing} if for any $x\in X$, there exists $\lambda_0>0$ such that $\lambda x\in B$ whenever $\abs{\lambda}\leq \lambda_0$. Balanced sets and absorbing sets all contain $0$.
Say that a topology $\tau$ on a subset $B$ of a real vector space is \emph{locally convex at a point $x\in B$} if $\tau$ has a local basis at $x$ consisting of convex sets and that $\tau$ is \emph{locally convex on $B$} if it is locally convex at every point of $B$.
A linear functional $\phi$ on a topological vector space $(X,\tau)$ is said to be \emph{$\tau$-bounded} if $\sup_{x\in V}\abs{\phi(x)}<\infty$ for some open set $V$.
We refer to \cite{R:91} for basic facts and undefined terminology on topological vector spaces.

\section{Main Results}\label{seclhb}

We prove the following result that is slightly stronger than Theorem~\ref{localHB}.

\begin{thm}\label{slocalHB}
Let $(X,\tau)$ be a real, Hausdorff topological vector space. Let $B$ be a convex, absorbing set in $X$ such that the relative topology $\tau_B$ is locally convex at $0$. Then for any $x\in X\bs \{0\}$, there exists a linear functional $\phi:X\to \R$ such that $\phi(x)\neq 0$ and that $\phi_{|B}$ is $\tau_B$-continuous at $0$.  Furthermore, if $B$ is also balanced, then $\phi_{|B}$ is $\tau_B$-continuous on $B$.
\end{thm}

\begin{proof}
Take $x\in X\bs\{0\}$. Since $\tau$ is Hausdorff and linear, there exists a $\tau$-open neighborhood $V_0$ of $0$ such that $(V_0+V_0)\cap (x+V_0+V_0)=\emptyset$. Inductively, take $\tau$-open sets $(V_k)$ containing $0$ such that $kV_k+V_k\subset V_{k-1}$ for any $k\in\N$. Then for any $k\geq 1$,
\begin{align}
\nonumber&(V_0+V_1+2V_2+\dots+kV_k+V_k)\cap (x+V_0+V_1+2V_2+\dots+kV_k+V_k)\\
\label{equ1}\subset&(V_0+V_0)\cap (x+V_0+V_0)\\
\nonumber=&\emptyset.
\end{align}
Since $\tau_B$ is locally convex at $0$, we can take a convex $\tau_B$-neighborhood $C_k$ of $0$ such that $C_k\subset V_k\cap B$ for each $k\in\N$.
Set $$D_k= C_1+ 2C_2 + \cdots + kC_k$$ for all $k\in \N$.
Since $0\in C_{k+1}$, $D_k \subset D_{k+1}$, for each $k\in\N$.
It follows in particular that $$D := \bigcup_{k=1}^\infty (D_k-D_k)$$ is a convex, balanced set.
Take a $\tau$-open set $U$ containing $0$ such that $U\cap B\subset C_1$. Since both $U$ and $B$ are absorbing, $U\cap B$, and therefore $C_1$, is also absorbing. It follows that $D$ is absorbing. Furthermore, since $D_k\subset V_1+2V_2+\dots+kV_k$, $D_k\cap (x+D_k)=\emptyset$ by \eqref{equ1}. Thus $x\notin D_k-D_k$ for all $k\in\N$, so that $x\notin D$.

Define the Minkowski functional of $D$ by $$\rho(y)=\inf\{\lambda>0:\lambda^{-1} y\in D\},\quad y\in X.$$
Then $\rho$ is a seminorm on $X$ and $\rho(x) \geq 1$ (see, e.g., \cite[Theorem~1.35]{R:91}). Define $\phi_0: \Span\{x\} \to \R$ by $\phi_0(\al x) = \al$.
Then $\phi_0$ is a linear functional on $\Span\{x\}$ and $\phi_0(\alpha x)=\alpha\leq \abs{\alpha}\leq \rho(\alpha x)$ for any $\alpha\in\R$.
By the vector-space version of Hahn-Banach Theorem (see, e.g., \cite[Theorem~3.2]{R:91}), $\phi_0$ extends to a linear functional $\phi: X\to \R$ such that $\phi(y) \leq \rho(y)$ for all $y\in X$. Clearly, $\phi(x)=\phi_0(x)=1$.
Suppose that $y\in C_k$ for some $k$.  Since $0\in D_k$, we have
\[ \pm ky \in \pm k C_k \subset D_k -D_k \subset D.\]
Therefore,
$$\pm k\phi(y)=\phi(\pm ky)\leq \rho(\pm ky)\leq 1,$$
so that $\abs{\phi(y)}\leq\frac{1}{k}$. Since each $C_k$ is a $\tau_B$-neighborhood of $0$ and $\phi_{|B}(0)=0$,  it follows that $\phi_{|B}$ is $\tau_B$-continuous at $0$.

Suppose that $B$ is also balanced. Take a net $(x_\alpha)\subset B$ and $x\in B$ such that $x_\alpha\stackrel{\tau}{\longrightarrow}x$. Then $B\ni\frac{x_\alpha-x}{2}\stackrel{\tau}{\longrightarrow}0$, and thus $\phi(\frac{x_\alpha-x}{2})\longrightarrow0$.
It follows that $\phi_{|B}(x_\alpha)\longrightarrow\phi_{|B}(x)$. This proves that $\phi_{|B}$ is $\tau_B$-continuous on $B$.
\end{proof}

A similar proof yields the following enhancement of  Theorem~\ref{slocalHB}.

\begin{thm}\label{sslocalHB}
Let $X$ be a real  vector space and $\tau^1\subset \tau^2$ be two Hausdorff linear topologies on $X$ such that $\tau^2$ is locally convex. Let $B$ be a convex set in $X$ containing $0$ such that the relative $\tau^1$-topology, $\tau^1_B$, is locally convex at $0$. Then for any $x\in B\bs \{0\}$, there exists a $\tau^2$-bounded linear functional $\phi:X\to \R$ such that $\phi(x)\neq 0$ and that $\phi_{|B}$ is $\tau^1_B$-continuous at $0$.  Furthermore, if $B$ is balanced, then $\phi_{|B}$ is $\tau^1_B$-continuous on $B$.
\end{thm}

\begin{proof}
We sketch the proof. Proceed as in the proof of Theorem~\ref{slocalHB} to construct $V_k$, $k\geq 0$, using the topology $\tau^1$. Since $V_0$ is also $\tau^2$-open and $\tau^2$ is locally convex, we can take a convex $\tau^2$-neighborhood $V$ of $0$ such that $V\subset V_0$. Now set $$D_k=V+C_1+2C_2+\cdots+ kC_k$$ for $k\in\N$.
Define $D$ in the same way. Note that since $V$ is absorbing and $0\in C_k$ for $k\in\N$, $D$ is also absorbing. Let $\phi_0$ and $\rho$ be as in the proof of Theorem~\ref{slocalHB}  and let $\phi$ be a linear extension of $\phi_0$ dominated by $\rho$.    Since $\pm V\subset D$, $\phi(\pm y)\leq \rho(\pm y)\leq 1$ and thus $\abs{\phi(y)}\leq 1$ for $y\in V$. Thus $\phi$ is $\tau^2$-bounded.  As before, one can show the $\tau^1_B$-continuity of $\phi_{|B}$ at $0$ and on $B$ if $B$ is balanced.
\end{proof}

\begin{remark}
\begin{enumerate}
\item
Theorems~\ref{slocalHB} and \ref{sslocalHB} (and therefore, Theorem~\ref{localHB}) also hold for complex vector spaces $X$. One simply regard $X$ as a real space and produce the real-linear functional $\phi'$ as in these theorems. Now set $\phi$ to be the complexication of $\phi'$: $\phi(x)=\phi'(x)-i\phi'(ix)$, for every $x\in X$.
\item
When $B$ is balanced, the restriction $\phi_{|kB}$ of the functionals $\phi$ produced in Theorems~\ref{slocalHB} and \ref{sslocalHB}  are $\tau_{kB}$- and  $\tau^1_{kB}$-continuous respectively for all $k\in \N$.
\end{enumerate}
\end{remark}

We now apply these results to the study of the uo-dual of Banach function spaces and Banach lattices.
For the rest of the paper, fix a  non-atomic probability space $(\Omega,\Sigma,\bP)$. Let $L^0:=L^0(\Omega,\Sigma,\bP)$ be the space of all real-valued measurable functions modulo a.s.-equality. $L^0$ carries a natural metrizable linear topology, namely, the topology of convergence in probability. We  call it the \emph{$L^0$-topology}, and to be concrete, we use the metric $d(x,y)=\int \abs{x-y}\wedge \one\mathrm{d}\bP$ to generate it.
A probability measure $\mathbb{Q}$ on $\Sigma$ is {\em equivalent} to $\bP$ (written as $\mathbb{Q} \sim \bP$) if
$\bP$ and $\mathbb{Q}$ are mutually absolutely continuous with respect to one another.
If $\mathbb{Q}\sim \bP$ then the $L^0(\mathbb{Q})$- and the $L^0(\bP)$-topologies are the same. A \emph{normed function space}  $X$ over $(\Omega,\Sigma,\bP)$ is a nonzero vector subspace of $L^0$ with a norm $\norm{\cdot}$ such that $y\in X$ and $\norm{y}\leq \norm{x}$ whenever $x,y\in L^0$ satisfy $x\in X$ and $\abs{y}\leq \abs{x}$ a.s.
The \emph{uo-dual} of $X$ is the collection of all linear functionals $\phi$ on $X$ such that $\phi(x_n)\longrightarrow 0$ whenever $(x_n)$ is norm bounded and $x_n\stackrel{a.s.}{\longrightarrow}0$.
Replacing a.s.-convergence with convergence in probability would result in the same collection of functionals.
Each $\phi\in X_{uo}^\sim$ can be represented by a suitable function $f\in L^0$ such that $\int\abs{fx}<\infty $ for all $x\in X$ via $\phi(x)=\la f,x\ra:=\int fx$ for all $x\in X$.
A linear functional $\phi$ on $X$ is {\em positive} if $\phi(x) \geq 0$ whenever $x \geq 0$.
$\phi$ is {\em strictly positive} if  $\phi(x) >0$ for any nonzero $x$ such that $x \geq 0$.
If $\phi\in X_{uo}^\sim$ is positive, then a representative $f\in L^0$ of it as above may be chosen to be a nonnegative function. In fact, by positivity one sees that $ \one_{\{f< 0\}}x=0$ a.s.~for all $x\in X$, and thus we can replace $f$ with $f\one_{\{f\geq 0\}}$ if necessary.

\begin{thm}\label{t8}
For a  normed function space $X$ over $(\Omega,\Sigma,\bP)$, the following statements are equivalent.
\begin{enumerate}
\item\label{t80} The relative $L^0$-topology on the unit ball of $X$ is locally convex.
\item\label{t81} The relative $L^0$-topology on the unit ball of $X$ is locally convex at $0$.
\item\label{t82} The uo-dual $X_{uo}^\sim$ separates points of $X$.
\item\label{t83} $X_{uo}^\sim$ admits a strictly positive member $\phi$.
\end{enumerate}
\end{thm}

\begin{proof}
The implication \eqref{t80}$\implies$\eqref{t81} is trivial. Let $\tau$ be the relative $L^0$-topology on $X$ and $B$ be the  unit ball of $X$. Let $X_B^*$ be the collection of functionals as defined in Theorem~\ref{localHB}. Clearly, $X_B^*=X_{uo}^\sim$. Thus the implication \eqref{t81}$\implies$\eqref{t82} follows immediately from Theorem~\ref{localHB}.

Suppose that \eqref{t82} holds.
As mentioned above, we identify functionals in $X_{uo}^\sim$ with certain functions in $L^0$.
Since $L^0$ is order complete and has the countable sup property, $\sup\{f\wedge \one: 0\leq f\in X_{uo}^\sim\}$ exists in $L^0$ and
there is an increasing nonnegative sequence $(f_n)$ in  $X_{uo}^\sim$ such that
\begin{equation}\label{supfn}
\sup\{f\wedge \one: 0\leq f\in X_{uo}^\sim\}=\sup\{f_n\wedge \one: n\in\N\}\quad\mbox{ in } L^0.\end{equation}
Since norm null sequences in $X$ are $L^0$-null (cf.~\cite[Proposition~2.6.3]{MN:91}), one sees that $X_{uo}^\sim$  is a norm closed subspace of the norm dual $X^*$ of $X$.
Put $f_0=\sum_{n\geq 1}\frac{f_n}{2^n\norm{f_n}_{X^*}+1}$. Then $f_0\in X^\sim_{uo}$.
We show that $f_0$ is a strictly positive element in $X_{uo}^\sim$.
Let $x$ be a nonzero positive element in $X$. Since $X_{uo}^\sim$ separates points of $X$, there exists $g
\in X^\sim_{uo}\bs\{0\}$ such that $\la g,x\ra\neq 0$.
Let  $f  =(\abs{g}\wedge \one)\one_{\{x>0\}}$. One sees that $0 \leq f\leq \one$, $f\in X^\sim_{uo}$ and $\langle f, x\rangle > 0$.
By (\ref{supfn}) and Monotone Convergence Theorem,
\[ 0 < \langle f, x\rangle =\langle f\wedge \one  , x\rangle
\leq \sup_n\langle f_n\wedge \one,x\rangle.\]
Therefore, there exists $n_0$ such that $\langle f_{n_0},x\rangle \geq \langle f_{n_0}\wedge \one,x\rangle > 0$.
It follows that $\langle f_0,x\rangle >0$, as required. This proves \eqref{t82}$\implies$\eqref{t83}.

Suppose that \eqref{t83} holds. Let $\phi\in X_{uo}^\sim$ be  strictly positive on $X$.
We claim that the relative topology on $B$ induced by the lattice norm $\norm{\cdot}_\phi:=\phi(\abs{\cdot})$ is the same as the relative $L^0$-topology, from which \eqref{t80} clearly follows.
Indeed, on the one hand, since norm null sequences are $L^0$-null, $\norm{\cdot}_\phi$-convergence implies $L^0$-convergence.
On the other hand, for $(x_n)\subset B$ and $x\in B$, if $x_n\stackrel{L^0}{\longrightarrow}x$, then $(\abs{x_n-x})_n$ is bounded and $L^0$-converges to $0$. Thus $\norm{x_n-x}_\phi=\phi(\abs{x_n-x})\longrightarrow0$  since  $X_{uo}^\sim = X^*_B$.
This proves the claim and hence \eqref{t83}$\implies$\eqref{t80}.
\end{proof}

Note that Condition~\eqref{t81} in Theorem~\ref{t8} is equivalent to saying that whenever a sequence $(x_n)$ in the unit ball converges to $0$ in probability then any of its convex block sequence  $(\sum_{j=p_{n-1}+1}^{p_n}c_jx_j)_n$, $0=p_0<p_1<p_2<\cdots$, also converges to $0$ in probability.

An easy application of Theorem~\ref{t8} is as follows. It is known from \cite{GLX:18} that $(L^p)_{uo}^\sim=L^q$ for $1<p\leq \infty$ and $p^{-1}+q^{-1}=1$ and $(L^{1})_{uo}^\sim=\{0\}$. Thus the relative $L^0$-topology is locally convex on the unit ball of $L^p$, $1<p\leq \infty$, but not on the unit ball of $L^1$.

A more striking application of Theorem~\ref{t8} concerns measure-free characterizations of uniform integrability. A set $K$ is  \emph{bounded in $L^0(\mathbb{P})$}  if $\lim_n\sup_{f\in K}\mathbb{P}(\abs{f}>n)=0$. It is  \emph{$\bP$-uniformly integrable} if $\lim_{\bP(A)\rightarrow0}\sup_{x\in K}\int_A \abs{x}\mathrm{d}\bP=0$. A set $K$ in $L^0$ such that  $y\in K$ whenever $\abs{y}\leq \abs{x}$ for some $x\in K$ is said to be \emph{solid}.
Solid sets contain $0$. The following beautiful result was proved by Kardaras \cite{K:14} using sophisticated methods.

\begin{cor}[Kardaras]\label{t9}
Let $K \subseteq {L}^0(\mathbb{P})$ be convex, solid, and bounded in probability. The following are equivalent.
\begin{enumerate}
\item\label{t91} The relative $L^0$-topology on $K$ is locally convex at $0$.
\item\label{t92} The relative $L^0(\mathbb{Q})$- and $L^1(\mathbb{Q})$-topologies coincide on $K$ for some $\mathbb{Q}\sim \mathbb{P}$.
\item\label{t93} $K$ is $\mathbb{Q}$-uniformly integrable for some $\mathbb{Q} \sim \mathbb{P}$.
\end{enumerate}
\end{cor}

\begin{proof}
The implications \eqref{t93}$\implies$\eqref{t92}$\implies$\eqref{t91} are clear. Suppose that \eqref{t91} holds. Put $X=\{\lambda x:\lambda\in\R, x\in K\}$ and $\rho(x)=\inf\big\{\lambda>0:\frac{x}{\lambda}\in K\big\}$ for any $x\in X$. Since $K$ is convex,  solid and bounded in probability, it is straightforward to verify that $(X,\rho)$ is a normed function space. The unit ball $B = \{x\in X:\rho(x)\leq 1\}$  is contained in $2K$.  Hence the relative $L^0$-topology on the unit ball $B$ is locally convex at $0$. By Theorem~\ref{t8}, $X_{uo}^\sim$ admits a strictly positive member $\phi$. Again, identify $\phi$ with some $0\leq f\in L^0$.
By strict positivity of $\phi$, one easily sees that $\one_{\{f= 0\}}y=0$ a.s.~for any $y\in X$. Set
\begin{align*}
g=\frac{f\one_{\{f>0\}}+\one_{\{f= 0\}}}{f+\one},\qquad \mathrm{d}\mathbb{Q}=\frac{g}{\int g\mathrm{d}\bP}\mathrm{d}\bP.
\end{align*}
Then $\mathbb{Q}$ is a probability measure equivalent to $\bP$.
Let $(x_n)\subset B$ and $x\in B$ be such that $x_n\stackrel{L^0}{\longrightarrow}x$. Then $\phi(\abs{x_n-x})\longrightarrow0$, so that
\begin{align}
\nonumber&\norm{x_n-x}_{L^1(\mathbb{Q})}\\
\label{equ3}=&\frac{1}{\int g\mathrm{d}\bP}\int \abs{x_n-x}\frac{f\one_{\{f>0\}}+\one_{\{f= 0\}}}{f+\one}\mathrm{d}\bP
\leq \frac{1}{\int g\mathrm{d}\bP}\int \abs{x_n-x}\big(f\one_{\{f>0\}}+\one_{\{f= 0\}}\big)\\
\nonumber=&\frac{1}{\int g\mathrm{d}\bP}\int \abs{x_n-x} f\mathrm{d}\bP
= \frac{1}{\int g\mathrm{d}\bP}\phi(\abs{x_n-x})\longrightarrow0.
\end{align}
If $K$ is not $\mathbb{Q}$-uniformly integrable, then there exist $\ep>0$, a sequence $(A_n)$ of measurable sets and a sequence $(y_n)\subset K$ such that $\mathbb{Q}(A_n)\longrightarrow0$ and $\int_{A_n}\abs{y_n}\mathrm{d}\mathbb{Q}\geq \ep$ for all $n\in\N$. Put $x_n=\one_{A_n}y_n$ for $n\in\N$. Since $\mathbb{Q}(A_n)\longrightarrow0$, $x_n\stackrel{L^0}{\longrightarrow}0$. Since $K$ is solid, $(x_n)\subset K\subset B$. But $\norm{x_n}_{L^1(\mathbb{Q})}\geq \ep$ for all $n\in\N$. This contradicts \eqref{equ3} and  proves \eqref{t91}$\implies$\eqref{t93}.
\end{proof}

Finally, we develop a Banach lattice version of Theorem~\ref{t8}.
We refer to \cite{AB:06} for facts and undefined terminology on vector and Banach lattices.
A net $(x_\alpha)$ in a vector lattice $X$ is said to {\em order
converge} to $x\in X$, written as $x_\al\stackrel{o}{\longrightarrow}x$,  if there is
another net $(y_\gamma)$ in $X$ such that $y_\gamma \downarrow 0$ and that for
every
$\gamma$, there exists $\al_0$ such that $\abs{x_\al - x} \leq y_\gamma$ for all
$\al \geq \al_0$. A net $(x_\al)$ in $X$ is said to {\em unbounded order
converge} ({\em uo-converge}) to $x\in X$ if $|x_\al-x| \wedge
y\stackrel{o}{\longrightarrow}0$ for any $y \in X_+$. In this case, we write
$x_\alpha\stackrel{uo}{\longrightarrow}x$. It is known that a sequence in a function space over $(\Omega,\Sigma,\bP)$ uo-converges to $0$ iff it converges to $0$ a.s. Moreover, $x_\alpha\stackrel{o}{\longrightarrow}0$ iff it has an order bounded tail and $x_\alpha\stackrel{uo}{\longrightarrow}0$.
For a Banach lattice $X$, the \emph{order continuous dual} $X_n^\sim$ (resp.~\emph{uo-dual} $X_{uo}^\sim$) is the collection of all linear functionals $\phi$ on $X$ such that $\phi(x_\alpha)\longrightarrow0$ whenever $x_\alpha\stackrel{uo}{\longrightarrow}0$ and $(x_\alpha)$ is order bounded (resp.~norm bounded) in $X$. \cite[Theorem~2.3]{GLX:18} characterizes $X_{uo}^\sim$ as the order continuous part of $X_n^\sim$.
We refer to \cite{GTX:17} and \cite{GLX:18} for facts about uo-convergence and uo-dual.
The uo-convergence has garnered interest as the natural generalization of a.s.-convergence.

For a Banach lattice $X$, \cite{DOT:17} introduced the \emph{un-topology}, which is a vector lattice topology on $X$, generated by the family of sets $$V_{y,\ep}=\big\{x\in X:\bignorm{\abs{x}\wedge y}<\ep\big\},\quad y\in X_+,\ep>0$$ as a local basis at $0$. The un-topology is generally not metrizable (see \cite{KMT:18,KT:18}). A net $(x_\alpha)$ un-converges to $x$ iff $\bignorm{\abs{x_\alpha-x}\wedge y}\longrightarrow0$ for any $y\in X_+$. A Banach lattice $X$ is said to be \emph{order continuous} if $\norm{x_\alpha}\longrightarrow0$ whenever $x_\al\stackrel{o}{\longrightarrow}0$  in $X$. \cite[Theorem~5.2]{KLT:18} asserts that the un-topology coincides with the topology of convergence in probability on an order continuous Banach function space. Thus the un-topology is regarded as a generalization of convergence in probability to (order continuous) Banach lattices.

By \cite[Theorem~5.2]{KMT:18}, the un-topology on an order continuous Banach lattice $X$ is never locally convex unless $X$ is atomic. We show that it can nevertheless be locally convex on the unit ball and the occurrence of this depends on how rich $X_{uo}^\sim$ is.

\begin{thm}\label{uod-oc}
Let $X$ an order continuous Banach lattice. The following are equivalent.
\begin{enumerate}
\item\label{uod-oc1} The relative un-topology on the unit ball is locally convex at $0$.
\item\label{uod-oc2} The uo-dual $X_{uo}^\sim$ separates points of $X$.
\end{enumerate}
\end{thm}

\begin{proof}To apply Theorem~\ref{localHB}, set $\tau$ to be the un-topology and $B$ be the unit ball. We claim that $X_B^*=X_{uo}^\sim$. Let $(x_\alpha)$ be a net in $B$ and $x\in B$ such that $x_\alpha\stackrel{uo}{\longrightarrow}x$. Then $|x_\al-x| \wedge y\stackrel{o}{\longrightarrow}0$, so that $\bignorm{\abs{x_\alpha-x}\wedge y}\longrightarrow0$, for any $y \in X_+$, by order continuity of $X$.
It follows that $X_{B}^*\subset X^\sim_{uo}$. For the reverse inclusion, suppose otherwise that there exists $\phi\in X_{uo}^\sim\backslash X_B^*$. Then $\phi(x_\alpha)\centernot{\longrightarrow} \phi(x)$ for some net $(x_\alpha)\subset B$ and $x\in B$ such that $x_\alpha\stackrel{un}{\longrightarrow}x$.
By passing to a subnet, we may assume that $\bigabs{\phi(x_\alpha)-\phi(x)}\geq \ep$ for some $\ep>0$ and all $\alpha$. Since $x_\alpha\stackrel{un}{\longrightarrow}x$, there exist countably many $(\alpha_n)$ such that $x_{\alpha_n}\stackrel{uo}{\longrightarrow}x$ (\cite[Corollary~3.5]{DOT:17}), so that $\phi(x_{\alpha_n})\longrightarrow\phi(x)$. This contradiction proves the claim. The implication \eqref{uod-oc1}$\implies$\eqref{uod-oc2} now follows immediately from Theorem~\ref{localHB}.

Suppose that \eqref{uod-oc2} holds but \eqref{uod-oc1} fails. Then there exist $y_0\in X_+$ and $\ep_0>0$ such that $V:=\big\{x\in B:\bignorm{\abs{x}\wedge y_0}<\ep_0\big\}$ does not contain a convex relative un-neighborhood of $0$ on $B$.
In particular, for any $n\in\N$,
\begin{align*}
\text{the convex hull of } \Big\{x\in B:\bignorm{\abs{x}\wedge y_0}<\frac{1}{n}\Big\}\text{ is not contained in }V.
\end{align*}
Let $E$ be the band generated by $y_0$ and $P$ be the corresponding band projection. Set $B^E=B\cap E$, the unit ball of $E$. Put  $V^E:=\big\{x\in B^E:\bignorm{\abs{x}\wedge y_0}<\ep_0\big\}$. Note that $\bignorm{\abs{x}\wedge y_0}=\bignorm{P(\abs{x}\wedge y_0)}=\bignorm{(P\abs{x})\wedge y_0}=\bignorm{\abs{Px}\wedge y_0}$ for all $x\in X$. Thus it is straightforward to verify that, for any $n\in\N$,
\begin{align}\label{localcon}
\text{the convex hull of } \Big\{x\in B^E:\bignorm{\abs{x}\wedge y_0}<\frac{1}{n}\Big\}\text{ is not contained in }V^E.
\end{align}
Since $y_0$ is a quasi-interior point of $E$, $\big\{x\in E:\bignorm{\abs{x}\wedge y_0}<\frac{1}{n}\big\}$, $n\in\N$, is in fact a local basis of the un-topology of $E$ at $0$ (cf.~\cite[Theorem~3.2]{KMT:18}).
Therefore, by \eqref{localcon}, the relative un-topology of $E$ to $B^E$ is not locally convex at $0$. That is, Condition~\eqref{uod-oc1} fails for $E$. Splitting $X=E\oplus E^d$, one sees that $E^\sim_{uo}$ separates points of $E$, i.e., Condition~\eqref{uod-oc2} holds for $E$. It is well-known that, being an order continuous Banach lattice with a quasi-interior point, $E$ is lattice isometric to an order continuous Banach function space $\widetilde{E}$ over a probability space (cf.~\cite[Proposition~1.b.14]{LT:79}). By Theorem~\ref{t8}, Conditions~\eqref{uod-oc1} and \eqref{uod-oc2} are equivalent for $\widetilde{E}$ and thus for $E$. This contradiction completes the proof.
\end{proof}

\end{document}